\documentclass[12pt, twoside, letterpaper, aps]{article}
\usepackage{amssymb,latexsym}
\usepackage{amsmath,amsfonts,amsthm}
\usepackage{paralist}
\usepackage{psfrag}
\usepackage[latin1]{inputenc}
\usepackage[italian,english]{babel}
\usepackage{paralist}
\usepackage{calc}
\newcommand{\Lim}[1]{\raisebox{0.6ex}{\scalebox{0.9}{$\displaystyle \lim_{#1}\;$}}}
\newcommand{\LimInf}[1]{\raisebox{0.6ex}{\scalebox{0.9}{$\displaystyle \liminf_{#1}\;$}}}
\newcommand{\vertiii}[1]{{\left\vert\kern-0.25ex\left\vert\kern-0.25ex\left\vert #1
    \right\vert\kern-0.25ex\right\vert\kern-0.25ex\right\vert}}
\usepackage{fancyhdr}
\def\latex/{\protect\LaTeX{}}
\def\tex/{\TeX}
\def\ams/{\protect\pAmS}
\def\pAmS{{\the\textfont2
        A\kern-.1667em\lower.5ex\hbox{M}\kern-.125emS}}
\def\amslatex/{\ams/-\latex/}

\input cyracc.def

\newfont{\tricyr}{wncyr10 at 12pt}
\newfont{\tricyi}{wncyi10 at 12pt}
\newfont{\tricyb}{wncyb10 at 12pt}
\newfont{\Tricyr}{wncyr10 at 13.6pt}
\newfont{\Tricyi}{wncyi10 at 13.6pt}
\newfont{\Tricyb}{wncyb10 at 13.6pt}
\newfont{\tricmr}{cmr10 at 13.6pt}
\newfont{\tricmi}{cmti10 at 13.6pt}
\newfont{\tricmb}{cmb10 at 13.6pt}

\newtheorem{thm}{Theorem}[section]

\newtheorem{lem}[thm]{Lemma}

\theoremstyle{definition}

\theoremstyle{remark}

\numberwithin{equation}{section}

\begin{document}

\title{\bf Existence and stability of standing waves for nonlinear Schr\"{o}dinger systems involving the fractional Laplacian}
\author{SANTOSH BHATTARAI}

\date{}

\maketitle

\begin{abstract}
In the present paper we consider the coupled system of nonlinear Schr\"{o}dinger equations with the fractional Laplacian
\[
\left\{
\begin{aligned}
  & \left(-\Delta\right)^\alpha u_1 = \lambda_1u_1+f_1(u_1)+\partial_1F(u_1,u_2)\ \textrm{in}\ \mathbb{R}^N,\\
   & \left(-\Delta\right)^\alpha u_2 = \lambda_2u_2+f_2(u_2)+\partial_2F(u_1,u_2)\ \textrm{in}\ \mathbb{R}^N,\\
\end{aligned}
\right.
\]
where $u_1, u_2:\mathbb{R}^N\to \mathbb{C}, N\geq 2,$ and $0<\alpha<1.$
By studying an appropriate family of constrained minimization problems, we obtain the existence of solutions
in $H^\alpha(\mathbb{R}^N) \times H^\alpha(\mathbb{R}^N)$ satisfying
\[
\int_{\mathbb{R}^N}|u_1|^2\ dx = \sigma_1\ \ \textrm{and}\ \ \int_{\mathbb{R}^N}|u_2|^2\ dx=\sigma_2
\]
for given $\sigma_j>0.$ The numbers $\lambda_1$ and $\lambda_2$ in the system appear
as Lagrange multiplier. The method
is based on the concentration compactness arguments,
but introduces a new way to verify some of the
properties of the
variational problem that are required in order for
the concentration compactness
method to work.
We consider the case when $f_j(s)=\mu_j|s|^{p_j-2}s$ and $F(s,t)=\beta |s|^{r_1}|t|^{r_2}$ with $\mu_j>0, \beta>0,$ and the values $r_i>1, 2<p_j, r_1+r_2<2+\frac{4\alpha}{N}.$
The method also enables us to prove
the stability result of standing wave solutions associated with the set of global minimizers.
\end{abstract}

\section{Introduction}

The present study is concerned with the coupled nonlinear fractional Schr\"{o}dinger system of the form
\begin{equation}\label{Fnls}
\left\{
\begin{aligned}
  & i \partial_t \Psi_1+\left(-\Delta\right)^\alpha \Psi_1 = f_1\left(\Psi_1\right)+\partial_1F(\Psi_1,\Psi_2),\ x\in \mathbb{R}^N,\ t>0,\\
   & i\partial_t \Psi_2+\left(-\Delta\right)^\alpha\Psi_2 = f_2\left(\Psi_2 \right)+\partial_2F(\Psi_1,\Psi_2),\ x\in \mathbb{R}^N,\ t>0,\\
\end{aligned}
\right.
\end{equation}
where $\Psi_1,\Psi_2:\mathbb{R}^+\times \mathbb{R}^N \to \mathbb{C},$ $N\geq 2, \alpha\in(0,1)$ is the fractional parameter, the function $f_j\in C\left(\mathbb{R}^N\times \mathbb{C}\right)$ satisfies $f_j(x, ze^{i\theta})=e^{i\theta}f_j(z),$ and $F\in C\left(\mathbb{R}^N\times \mathbb{C}^2\right)$ satisfies
\[
\partial_jF(x,z_1e^{i\theta_1}, z_2e^{i\theta_2})=e^{i\theta_j}\partial_jF(z_1, z_2),\ \ j=1,2.
\]
For any parameter $\alpha\in (0,1),$ the fractional Laplacian $\left(-\Delta \right)^\alpha$
is defined via Fourier transform as
\begin{equation}\label{FLdef}
\widehat{(-\Delta)^\alpha u}(\xi)=|\xi|^{2 \alpha}\widehat{u}(\xi),\  u\in \mathcal{S}(\mathbb{R}^N),
\end{equation}
where $\mathcal{S}(\mathbb{R}^N)$ denotes the Schwartz space of rapidly decaying $C^\infty(\mathbb{R}^N)$ functions.

\smallskip

Nonlinear fractional Schr\"{o}dinger equation was
introduced by N. Laskin in a series of papers \cite{Las1, Las2, Las3} by generalizing
the Feynman path integral over Brownian-like paths to L\'{e}vy-like quantum paths$.$ In other words,
if the Feynman path integral over Brownian trajectories
allows one to reproduce the well known NLS equation, then the path
integral over L\'{e}vy trajectories leads one to a space-fractional
Schr\"{o}dinger equation (see \cite{Las3}).
The models involving the fractional Laplacian arise
in the description of a wide variety of phenomena
in the applied sciences such as finance, plasma physics, obstacle problems, semipermeable membrane, anomalous diffusion, to name a few.
For instance, the reader may consult \cite{Kir}
for a rigorous derivation of fractional NLS type equations
starting from a family of models
for charge transport in
biopolymers like the deoxyribonucleic acid (DNA) and \cite{Met2, Met} for the
derivation of many fractional differential equations asymptotically from L\'{e}vy random walk models.
Research in the fractional Schr\"{o}dinger equations have recently begun to receive more attention and
the literature for this research area is still expanding and rather young in the mathematics realm.

\smallskip

In this paper we study existence and stability of standing wave solutions of the system \eqref{Fnls}.
A standing wave for the system \eqref{Fnls} is a solution of the
form $\Psi_j(t,x)=e^{i\lambda_jt}u_j(x)$ for some numbers $\lambda_1,\lambda_2\in\mathbb{R}.$
Plugging the standing wave ansatz into \eqref{Fnls}, one sees that the functions $u_1, u_2:\mathbb{R}^N\to \mathbb{C}$
satisfy the following coupled system of time independent equations
\begin{equation}\label{ODE}
\left\{
\begin{aligned}
  & \left(-\Delta\right)^\alpha u_1 = \lambda_1 u_1+ f_1(u_1)+\partial_1F(u_1,u_2)\ \ \textrm{in}\ \mathbb{R}^N,\\
   & \left(-\Delta\right)^\alpha u_2 = \lambda_1 u_2+ f_1(u_2)+\partial_2F(u_1,u_2) \ \ \textrm{in}\ \mathbb{R}^N.
\end{aligned}
\right.
\end{equation}
The study of standing wave solutions is of particular interest in physics.
The question of existence of solutions has been well studied in the literature
for the standard nonlinear Schr\"{o}dinger type equations and their coupled versions.
The approach to stability theory taken here will be the variational approach.
The orbital stability of standing waves for the standard nonlinear
Schr\"{o}dinger equations was proved by Cazenave and Lions \cite{CLi} by using
the concentration compactness principle \cite{Lio, Lio2}
to characterize these special waves as minimizers of certain
variational problems. There are also many results concerning the stability of standing waves for coupled systems of
nonlinear Schr\"{o}dinger equations (for example, see \cite{ABnls, SBdcds, JeanJ, NWang}).
Concerning the fractional Schr\"{o}dinger equations, Guo and Huang \cite{Guo} have recently proved
the stability of standing waves by using Cazenave and Lions's concentration compactness arguments.
In the present paper we prove results for
coupled fractional Schr\"{o}dinger systems which are in the same spirit of the results of the above cited papers.
To our knowledge, none has addressed the questions of existence and stability of solutions
with prescribed $L^2$-norms 
for coupled systems of fractional Schr\"{o}dinger equations.

\smallskip

In the study of nonlinear systems such as \eqref{ODE}, it is important to choose proper function space.
In this work, we consider \eqref{ODE} in the energy space
$
H^\alpha(\mathbb{R}^N)\times H^\alpha(\mathbb{R}^N),
$
with the norm denoted by $\|\cdot\|_{\alpha}.$
Here and in what follows $H^\alpha(\mathbb{R}^N)$ denotes the fractional order Sobolev space
\[
H^\alpha(\mathbb{R}^N)=\left\{u\in L^2(\mathbb{R}^N): \xi\mapsto (1+|\xi|^2)^{\alpha/2}\widehat{u}(\xi) \in L^2(\mathbb{R}^N, d\xi) \right\}
\]
with the norm $\|\cdot\|_{H^\alpha(\mathbb{R}^N)}$ given by
$
\|u\|_{H^\alpha(\mathbb{R}^N)}=\|(1+|\xi|^2)^{\alpha/2}\widehat{u}(\xi)\|_{L^2(\mathbb{R}^N)}.$
We shall denote by $\|\cdot\|_{L^p(\mathbb{R}^N)}$ the norm in the space $L^p(\mathbb{R}^N).$
The function spaces appearing in this paper will, unless otherwise stated, all have complex-valued functions.

\smallskip

Following closely the strategies of
\cite{AAkdv, ABnls, [AB11], SBdcds}, we first address the existence question
of standing wave profiles $(u_1,u_2)$ in $H^\alpha(\mathbb{R}^N)\times H^\alpha(\mathbb{R}^N)$ satisfying
\begin{equation}\label{constraints}
\int_{\mathbb{R}^N}|u_1|^2\ dx=\tau_1 \ \ \textrm{and}\ \ \int_{\mathbb{R}^N}|u_2|^2\ dx=\tau_2
\end{equation}
for given $\tau_1>0$ and $\tau_2>0.$ In literature these special solutions are also called $L^2$ normalized (or simply normalized) solutions.
To avoid technicalities, we only consider the case when $f_1(s)=\mu_1|s|^{p_1-2}s, f_2(s)=\mu_2|s|^{p_2-2}s,$ and $F(s,t)=\beta |s|^{r_1}|t|^{r_2}$ and assume throughout that the following assumptions hold:
\begin{equation}\label{assump}
N\geq 2,\ 0<\alpha<1, \mu_j>0, \beta>0, r_i>1,\ \textrm{and}\ 2<p_1,p_2, r_1+r_2<2+\frac{4\alpha}{N}.
\end{equation}

 A standard way often used to describe standing wave solutions
is as critical points of constrained variational problems, in which the functional being
minimized and the constraint functionals are conserved quantities. For purposes of
investigating the stability of standing waves, however, one needs to characterize them not just as critical points,
but as absolute minimizers of constrained variational problems.
In this paper, we consider, for any $\tau=(\tau_1,\tau_2)\in \mathbb{R}^+\times \mathbb{R}^+,$
the problem of finding minimizers of the energy
\begin{equation*}
\begin{aligned}
E(u_1,u_2)& =\frac{1}{2}\int_{\mathbb{R}^N}\left(|\left(-\Delta \right)^{\alpha/2}u_1|^2 + |\left(-\Delta \right)^{\alpha/2}u_2 |^2 \right)\ dx \\ & -\int_{\mathbb{R}^N}\left(\frac{\mu_1}{p_1}|u_1|^{p_1}+\frac{\mu_2}{p_2}|u_2|^{p_2}+F(u_1,u_2)\right)\ dx
\end{aligned}
\end{equation*}
over the set $\mathcal{S}_\tau=S_{\tau_1}\times S_{\tau_2},$
where $S_r=\{u\in H^\alpha(\mathbb{R}^N):\|u\|_{L^2(\mathbb{R}^N)}=\sqrt{r}~\}$ for any $r>0.$
If there exists a minimizer
$(u_1,u_2)$ for the
problem $(E, \mathcal{S}_\tau),$ then it is a solution of \eqref{ODE}
and the function $(\Psi_1, \Psi_2)$ defined
by $\Psi_j(t,x)=e^{i\lambda_j t}u_j(x)$ is a
standing wave solution of \eqref{Fnls}.
The numbers $\lambda_1$ and $\lambda_2$ are the Lagrange multipliers
associated to the stationary point $(u_1,u_2)$ on $\mathcal{S}_\tau.$
The stability of the set of minimizers follows by a standard principle
since both the energy $E(u_1,u_2)$ and the constraint
functionals $\int_{\mathbb{R}^N}|u_i|^2\ dx$ are conserved along the flow of \eqref{Fnls}.
To prove the precompactness of an energy-minimizing sequence via
the method of concentration compactness,
we establish certain strict inequalities of the function involving
the infimum of the problem $(E, \mathcal{S}_\tau)$ as a function of two
parameters $\tau_1$ and $\tau_2.$

\smallskip

Our first main theorem addresses the existence of minimizers of $(E, \mathcal{S}_\tau).$

\begin{thm}
Suppose that the assumptions \eqref{assump} hold. Then

\smallskip

(i) for any given $\tau \in \mathbb{R}^{+}\times\mathbb{R}^{+},$ there exists a nonempty
subset $V(\tau)$ of $H^\alpha(\mathbb{R}^N)\times H^\alpha(\mathbb{R}^N)$ consisting of
minimizers of the problem $(E, \mathcal{S}_\tau).$

\smallskip

(ii) any minimizer $(u_1,u_2)\in V(\tau)$
satisfies \eqref{ODE}, for some $\lambda \in \mathbb{R}^{+}\times\mathbb{R}^{+},$ and the function
$(\Psi_1, \Psi_2)$ defined by $\Psi_j(t,x)=e^{i\lambda_j t}u_j(x)$
is a standing wave solution of \eqref{Fnls}
satisfying
\[
\|\Psi_1\|_{L^2(\mathbb{R}^N) }=\sqrt{\tau_1}\ \ \textrm{and}\ \ \|\Psi_2\|_{L^2(\mathbb{R}^N)}=\sqrt{\tau_2}.
\]

\smallskip

(iii) if $\{(u_1^n,u_2^n) \}_{n\geq 1}$ is an energy-minimizing sequence for the problem $(E,\mathcal{S}_\tau),$
then there exists a sequence of points $\{y_k\}\subset \mathbb{R}^N$ and a subsequence $\{(u_1^{n_k},u_2^{n_k}) \}_{k\geq 1}$
such that $(u_1^{n_k}(\cdot+y_k),u_2^{n_k}(\cdot+y_k))\to (u_1,u_2)$ strongly
in $H^\alpha(\mathbb{R}^N)\times H^\alpha(\mathbb{R}^N),$ as $k\to \infty,$ where $(u_1,u_2)$ is some minimizer for $(E,\mathcal{S}_\tau).$
Moreover, $(u_1^n,u_2^n)\to V(\tau)$ in the following sense
\[
\lim_{n\to \infty}\inf_{(u_1,u_2)\in V(\tau)}\|(u_1^n,u_2^n)-(u_1, u_2) \|_{\alpha}=0.
\]

\smallskip

(iv) the family of sets $V(\tau)$ is mutually disjoint in the following sense
\[
\forall \tau, \sigma \in \mathbb{R}^{+}\times\mathbb{R}^{+}, \tau \neq \sigma \Rightarrow V(\tau)\cap V(\sigma)=\emptyset.
\]
\end{thm}

The following is our orbital stability result of solutions, which is a direct consequence of the
result of relative compactness.

\begin{thm}\label{stabilityTHM}
For any $\tau \in \mathbb{R}^{+}\times\mathbb{R}^{+},$ the set of standing wave profiles $V(\tau)$ is
stable, i.e., for every $\varepsilon>0$ there exists $\delta>0$ such that if the initial
condition $\Psi_0=(\Psi_{1, 0}, \Psi_{2, 0})\in H^\alpha(\mathbb{R}^N)\times H^\alpha(\mathbb{R}^N)$ satisfies
\[
\inf_{\mathbf{u}\in V(\tau)}\|\Psi_0-\mathbf{u} \|_{\alpha}<\delta,
\]
then any solution $\Psi(t)=(\Psi_1(t), \Psi_2(t))$ of the system \eqref{Fnls} emanating from
$\Psi_0$
satisfies
\[
\sup_{t\geq 0}\inf_{\mathbf{u}\in V(\tau)} \|\Psi(t)-\mathbf{u} \|_{\alpha}<\varepsilon.
\]
\end{thm}
Theorem~\ref{stabilityTHM} must however be understood in a
qualified sense because we lack a suitable global well-posedness theory for the
initial value problem.

\smallskip

The paper is organized as follows. In Section~\ref{preli} we provide some well-known
results about fractional Sobolev spaces and prove some preliminary lemmas
which play the most important roles in this paper.
Section~\ref{existence} proves the existence theorem and
the stability result is proved in Section~\ref{stability}.
Throughout this paper, the same letter $C$ might be used to denote
various positive constants which may take different values
within the same string of inequalities.

\section{Preliminary lemmas}\label{preli}

For the reader's convenience,
we first provide some basic properties of the fractional order Sobolev spaces $H^\alpha(\mathbb{R}^N)$ which will be
used throughout the paper.
The first lemma provides an alternative way of defining the fractional Sobolev space $H^\alpha(\mathbb{R}^N).$
\begin{lem}\label{equiHs}
Let $0<\alpha<1$ and $N\geq 1.$ Then
$u\in H^\alpha(\mathbb{R}^N)$ if and only if
\[
u\in L^2(\mathbb{R}^N)\ \textrm{and}\ \ (x,y)\to \frac{|u(x)-u(y)|}{|x-y|^{\frac{N}{2}+\alpha}}\in L^2\left(\mathbb{R}^{2N}, dxdy \right).
\]
Moreover, the norm $\|\cdot \|_{H^\alpha(\mathbb{R}^N)}$ in the space $H^\alpha(\mathbb{R}^N)$ is equivalent to
\[
\vertiii{u}_{H^\alpha(\mathbb{R}^N)}=\left(\|u\|_{L^2(\mathbb{R}^N)}^2+\left[u\right]^{2}_{H^\alpha(\mathbb{R}^N)} \right)^{1/2},
\]
where the quantity $ \left[u\right]^{2}_{H^\alpha(\mathbb{R}^N)}$ (so-called Gagliardo seminorm of $u$) is given by
\[
\left[u\right]^{2}_{H^\alpha(\mathbb{R}^N)}=\iint_{\mathbb{R}^N\times \mathbb{R}^N}\frac{|u(x)-u(y)|^2}{|x-y|^{N+2\alpha}}\ dxdy.
\]

\end{lem}
The proof of the following alternative definition of the fractional Laplacian operator can be found, for example, in \cite{NPV}.
\begin{lem}\label{LaEqi}
For any $0<\alpha<1,$ the fractional Laplacian operator $(-\Delta)^\alpha:\mathcal{S}(\mathbb{R}^N)\to L^2(\mathbb{R}^N)$ as
defined in \eqref{FLdef} can be expressed as
\[
\begin{aligned}
(-\Delta)^{\alpha}u(x) & =C(N,\alpha)~\textrm{P.V.} \int_{\mathbb{R}^N}\frac{u(x)-u(y)}{|x-y|^{N+2\alpha}}dy\\
 & = C(N,\alpha)\lim_{\varepsilon\to 0+}\int_{\mathbb{R}^N\setminus B_\varepsilon(x)}\frac{u(x)-u(y)}{|x-y|^{N+2\alpha}}dy,
 \end{aligned}
\]
where P.V. denotes abbreviation for the Cauchy principal value of the singular integral and $C(N,\alpha)>0$ is some normalization constant given by
\[
C(N,\alpha)=\left(\int_{\mathbb{R}^N}\frac{1-\cos x_1}{|x|^{N+2\alpha}}\ dx\right)^{-1}=\frac{\Gamma((2\alpha+N)/2)}{|\Gamma(-\alpha)|}2^{2\alpha-1}\pi^{-N/2}.
\]
\end{lem}
This above integral representation shows the nonlocal feature of the fractional Laplacian $(-\Delta)^\alpha u$ and
can be used to define the operator for more
general functions, for example, $u\in C^2(\mathbb{R}^N).$
The next lemma provides the relationship between the fractional Laplacian $(-\Delta)^{\alpha}$ and
the fractional Sobolev space $H^\alpha(\mathbb{R}^N)$ (for a proof, see Propositions 4.2 and 4.4 of \cite{NPV}).
\begin{lem}
Let $0<\alpha<1.$ Then for all $u\in H^\alpha(\mathbb{R}^N),$
\[
\left[u\right]^{2}_{H^\alpha(\mathbb{R}^N)}
= \left( \frac{C_{N,\alpha}}{2}\right)^{-1}\int_{\mathbb{R}^N}|\xi|^{2\alpha}|\widehat{u}(\xi)|^2\ d\xi=\left( \frac{C_{N,\alpha}}{2}\right)^{-1}\|(-\Delta)^{\alpha/2}u \|_{L^2(\mathbb{R}^N)}^2.
\]
\end{lem}

The following lemma is the Sobolev-type inequality for the fractional order Sobolev spaces.
An elementary proof of this is given in Theorem~6.5 of \cite{NPV}.
\begin{lem}\label{sobInq}
Let $0<\alpha<1$ be such that $N>2\alpha.$ Then there exists a positive constant $C=C(N,\alpha)$ such that
\[
\|u\|_{L^{2_\alpha^\star}(\mathbb{R}^N)}^2\leq C \left[u\right]_{H^\alpha(\mathbb{R}^N)}^2,\  \forall u\in H^\alpha(\mathbb{R}^N),
\]
where $2_\alpha^\star\equiv 2_\alpha^\star(N, \alpha)=\frac{2N}{N-2\alpha}$ is the fractional critical exponent and
$\left[u\right]^{2}_{H^\alpha(\mathbb{R}^N)}$
is the Gagliardo seminorm of $u$ as defined in Lemma~\ref{equiHs}.

\smallskip

Consequently, the space $H^\alpha (\mathbb{R}^N)$ is continuously embedded into
$L^q(\mathbb{R}^N)$ for any $q\in [2,2_\alpha^\star]$ and compactly embedded into $L_{\textrm{loc}}^q(\mathbb{R}^N)$
for any $q\in [2,2_\alpha^\star).$
\end{lem}

The next lemma is the fractional Gagliardo-Nirenberg inequality.
\begin{lem}
Let $1\leq p<\infty, 0<\alpha<1,$ and $N>2\alpha.$ Then for any $u\in H^\alpha(\mathbb{R}^N),$
\[
\|u\|_{L^p(\mathbb{R}^N)}\leq C \left[u \right]_{H^\alpha(\mathbb{R}^N)}^\lambda \|u\|_{L^q(\mathbb{R}^N)}^{1-\lambda},
\]
where $q\geq 1, \lambda\in [0,1],$ $C=C_{N,\alpha, \lambda}$ is a positive constant, and $\lambda$ satisfies
\[
\frac{N}{p}=\frac{\lambda(N-2\alpha)}{2}+\frac{N(1-\lambda)}{q}.
\]
\end{lem}
\begin{proof}
This is a consequence of the H\"{o}lder inequality and the Sobolev-type inequality (Lemma~\ref{sobInq}). The case $p=1$ is clear. For $p>1,$ using the H\"{o}lder inequality, one has
\[
\|u\|_{L^p(\mathbb{R}^N)}\leq \|u\|_{L^{2_{\alpha}^\star }(\mathbb{R}^N) }^{\lambda}\|u\|_{L^q(\mathbb{R}^N) }^{1-\lambda},\ \textrm{where}\ \ \frac{1}{p}=\frac{1-\lambda}{q}+\frac{\lambda}{2_{\alpha}^\star }.
\]
Now using the Sobolev inequality (Lemma~\ref{sobInq}), we obtain that
\[
\|u\|_{L^p(\mathbb{R}^N)}\leq (C_{N,\alpha})^{\lambda/2} \left[u \right]_{H^\alpha(\mathbb{R}^N) }^\lambda \|u\|_{L^q(\mathbb{R}^N) }^{1-\lambda},
\]
where $C_{N,\alpha}$ is the same constant as in Lemma~\ref{sobInq}.
The desired inequality follows by taking $C=(C_{N,\alpha})^{\lambda/2}$ in the last inequality.
\end{proof}

We now turn our attention to proving some properties of the variational problem $(E,\mathcal{S}_\tau)$ and its minimizing sequences,
which will be used in the proof of the existence theorem.
Throughout the rest of this paper, we shall use the following notation
\begin{equation}\label{singleEng}
e_i(u)=\frac{1}{2}\|D^\alpha u\|_{L^2(\mathbb{R}^N)}^2 -\frac{\mu_i}{p_i}\|u\|_{L^{p_i}(\mathbb{R}^N)}^{p_i}\ \textrm{for}\ i=1,2,
\end{equation}
where $D^\alpha=\left(-\Delta \right)^{\alpha/2}.$ We denote by $\gamma_1=\mu_1/ p_1$ and $\gamma_2=\mu_2/p_2$ the
constants appearing in \eqref{singleEng}.
For any $\beta=(\beta_1, \beta_2),$ we define the function $E_\beta$ by
\begin{equation}
E_\beta=\inf\left\{E(u_1,u_2):(u_1,u_2)\in \mathcal{S}_\beta\right\}.
\end{equation}
We first prove that $E_\tau$ is finite and negative.
\begin{lem}\label{NegInfE}
For any $\tau\in \mathbb{R}^{+}\times\mathbb{R}^{+},$ one has
$
-\infty<E_\tau<0.
$
\end{lem}
\noindent{\bf Proof.}\ The proof that $E_\tau<0$ follows a standard
scaling argument. Indeed, for a given $(u_1,u_2)\in \mathcal{S}_\tau,$ define $u_1^{\lambda}=\lambda^{1/2}u_1(\lambda^{1/N} x)$ and
$u_2^{\lambda}=\lambda^{1/2}u_2(\lambda^{1/N} x)$ for any $\lambda>0.$ Then, we have that $(u_1^{\lambda}, u_2^{\lambda})\in \mathcal{S}_\tau$ as well.
Since the fractional Laplacian $D^\alpha$ behaves like
differentiation of order $\alpha,$ i.e., $D^\alpha h_\theta(x)=\theta^{\alpha}D^\alpha h(\theta x)$ for any $\theta>0$ and $h_\theta(x)=h(\theta x),\ x\in \mathbb{R}^N,$ we obtain that
\[
E\left(u_1^{\lambda}, u_2^{\lambda} \right)\leq \frac{\lambda^{2\alpha/N}}{2} \int_{\mathbb{R}^N}\left(|D^\alpha u_1|^2+|D^\alpha u_2|^2\right)\ dx-\lambda^{\theta}\int_{\mathbb{R}^N}F(u_1,u_2)\ dx
\]
where $\theta=\frac{(r_1+r_2)-2}{2}<\frac{2\alpha}{N}$ by $r_1+r_2<2+\frac{4\alpha}{N}.$ Thus one can
take $\lambda>0$ sufficiently small such that $ E\left(u_1^{\lambda}, u_2^{\lambda} \right)<0$ and
consequently $E_\tau<E\left(u_1^{\lambda}, u_2^{\lambda} \right)<0.$

\smallskip

Next, making use of the H\"{o}lder inequality and the Sobolev inequality, we obtain
\begin{equation}\label{SobInq}
\begin{aligned}
\left(\int_{\mathbb{R}^N}|u_1|^{p_1}\ dx \right)^{1/p_1} & \leq \left(\int_{\mathbb{R}^N}|u_1|^{2}\ dx \right)^{(1-\lambda)/2}\left(\int_{\mathbb{R}^N}|u_1|^{2_\alpha^\star}\ dx \right)^{\lambda/2_\alpha^\star}\\
& \leq C\left(\int_{\mathbb{R}^N}|u_1|^{2}\ dx \right)^{(1-\lambda)/2} \vertiii{u_1}_{H^\alpha(\mathbb{R}^N)}^{\lambda} ,\
\end{aligned}
\end{equation}
where $\lambda=N(p_1-2)/2p_1 \alpha.$ We now use the Young inequality to obtain
\[
\int_{\mathbb{R}^N}|u_1|^{p_1}\ dx \leq \varepsilon \vertiii{u_1}_{H^\alpha(\mathbb{R}^N)}^2 + C_\varepsilon \left( \int_{\mathbb{R}^N}|u_1|^{2}\ dx\right)^\mu,\ \mu=\frac{(1-\lambda)p_1}{2-\lambda p_1},
\]
for sufficiently small $\varepsilon>0$ and $C_\varepsilon$ depends on $\varepsilon$ but not on $u_1.$
Similar inequality holds for $\int_{\mathbb{R}^N}|u_2|^{p_2}\ dx.$ Then, for any $(u_1,u_2)\in \mathcal{S}_\tau,$ one obtains that
\[
\begin{aligned}
E(u_1,u_2)& \geq \frac{1}{2}\vertiii{u_1}_{H^\alpha(\mathbb{R}^N)}^2+\frac{1}{2}\vertiii{u_2}_{H^\alpha(\mathbb{R}^N)}^2
-C\left(\|u_1\|_{L^{p_1}(\mathbb{R}^N) }^{p_1}+\|u_2\|_{L^{p_2}(\mathbb{R}^N)}^{p_2}\right)-\tau_3\\
& \geq \frac{1-\varepsilon}{2} \left(\vertiii{u_1}_{H^\alpha(\mathbb{R}^N)}^2+\vertiii{u_2}_{H^\alpha(\mathbb{R}^N)}^2 \right)-C,
\end{aligned}
\]
where $\tau_3= (\tau_1+\tau_2)/2$ and $C=C(N, p_1, p_2, \alpha, \varepsilon, \tau)$ are positive constants. Taking $2\varepsilon<1,$ this
last inequality shows that $E_\tau>-\infty$ holds for all $N<\infty.$ \hfill{$\Box$}

\smallskip

In the next lemma, we prove that any energy-minimizing sequence
for the problem $(E,\mathcal{S}_\tau)$ must be bounded in $H^\alpha(\mathbb{R}^N) \times H^\alpha(\mathbb{R}^N)$ and
collect some of their special properties.

\begin{lem}\label{BounMin}
Suppose $\{(u_1^{n},u_2^{n})\}_{n\geq 1}$ be any sequence of functions in $H^\alpha(\mathbb{R}^N)\times H^\alpha(\mathbb{R}^N)$ such that
\begin{equation}\label{mindef}
\lim_{n\to \infty}\|u_i^{n}\|_{L^2(\mathbb{R}^N)}=\sqrt{\tau_i}\ \ \textrm{and}\ \ \lim_{n\to \infty}E(u_1^{n},u_2^{n})=E_\tau.
\end{equation}
Then the following assertions hold:

\smallskip

(i) there exists $B>0$ such that $\vertiii{u_1^{n}}_{H^\alpha(\mathbb{R}^N)}+\vertiii{u_2^{n}}_{H^\alpha(\mathbb{R}^N)}\leq B$ for all $n.$

\smallskip

(ii) there exists $\delta_i>0$ and $N_{\delta_i}$ such that $\|u_i^n\|_{L^{p_i}(\mathbb{R}^N) }^{p_i} \geq \delta_i$ for all $n\geq N_{\delta_i}.$

\smallskip

(iii) for any $\lambda>1$ and for all sufficiently large $n,$ the energy $e_i(u)$ satisfies the following scaling property
\[
e_i\left(\lambda u_i^{n}\right)<\lambda^2 e_i\left(u_i^{n}\right)\ \textrm{for}\ i=1,2.
\]
\end{lem}
\noindent{\bf Proof.}\ For any minimizing sequence $\{(u_1^n,u_2^n)\}_{n\geq 1},$ it follows from \eqref{SobInq} that
\begin{equation}\label{Est1}
\int_{\mathbb{R}^N} |u_1^n|^{p_1}\ dx\leq C\vertiii{u_1^n}_{H^\alpha(\mathbb{R}^N)}^{N(p_1-2)/2\alpha}\leq C \vertiii{(u_1^n,u_2^n)}_{H^\alpha(\mathbb{R}^N)}^{N(p_1-2)/2\alpha},
\end{equation}
where $C=C(p_1, N, \tau_1, \alpha).$
Similar estimate holds for $\int_{\mathbb{R}^N} |u_2^n|^{p_2}\ dx.$
Using the H\"{o}lder inequality and the estimates for $\int_{\mathbb{R}^N} |u_1^n|^{p_1}\ dx$ and $\int_{\mathbb{R}^N} |u_2^n|^{p_2}\ dx$, we also have
\begin{equation}\label{Est3}
\begin{aligned}
\int_{\mathbb{R}^N} |u_1^n|^{r_1}|u_2^n|^{r_2}\ dx  \leq \left( \int_{\mathbb{R}^N} |u_1^n|^{r_1q}\right)^{1/q} \left( \int_{\mathbb{R}^N} |u_2^n|^{r_2q^\prime}\right)^{1/q^\prime}
\leq C \vertiii{(u_1^n,u_2^n)}_{H^\alpha(\mathbb{R}^N)}^{\mu_1+\mu_2 }
\end{aligned}
\end{equation}
where the exponents $\mu_1$ and $\mu_2$ are given by $\mu_1=N(r_1q-2)/2q\alpha$ and $\mu_2=N(r_2q^\prime-2)/2q^\prime \alpha$ with $1/q+1/q^\prime=1$ and $C=C(N, p_1, p_2, r_1, r_2, \tau, q, \alpha).$
Now put $F_n=(u_1^n,u_2^n)$ and observe that
\[
\frac{1}{2}\vertiii{F_n}_{H^\alpha(\mathbb{R}^N)}^2=E\left(F_n\right)+\gamma_1\|u_1^n\|_{L^{p_1}(\mathbb{R}^N)}^{p_1}
+\gamma_2\|u_2^n\|_{L^{p_2}(\mathbb{R}^N)}^{p_2}+\int_{\mathbb{R}^N}
F(F_n)\ dx+ \tau_3.
\]
Since the sequence $\{E\left(u_1^n, u_2^n\right) \}_{n\geq 1}$ is
bounded, we obtain that
\[
\frac{1}{2}\vertiii{F_n}_{H^\alpha(\mathbb{R}^N)}^2\leq C \left( \vertiii{F_n}_{H^\alpha}^{N(p_1-2)/2\alpha}+\vertiii{F_n}_{H^\alpha(\mathbb{R}^N)}^{N(p_2-2)/2\alpha}+\vertiii{F_n}_{H^\alpha(\mathbb{R}^N)}^{\mu_1+\mu_2 }\right)
\]
Since the exponent $N(p_i-2)/2\alpha$ belongs to $(0,2)$ for
any $p_i\in (2, 2+4\alpha/N)$ and $\mu_1+\mu_2<2,$ it follow that the
sequence $\{\left(u_1^n, u_2^n\right)\}_{n\geq 1}$ is bounded in $H^\alpha(\mathbb{R}^N)\times H^\alpha(\mathbb{R}^N).$

\smallskip

To prove that $L^{p_i}$-norms are bounded away from zero for large $n,$ we argue
by contradiction. If no such a positive number $\delta_1$ exists, then
$
\LimInf{n\to \infty}\int_{\mathbb{R}^N}|u_1^n|^{p_1}\ dx = 0.
$
Consequently, $\int_{\mathbb{R}^N}F(u_1^n, u_2^n)\ dx\to 0$ as $n\to \infty,$ and we have that
\begin{equation}\label{conlem}
E_\tau=\lim_{n\to \infty}E\left(u_1^n,u_2^n\right)\geq \liminf_{n\to \infty}e_2\left(u_2^n\right).
\end{equation}
On the other hand, select $\phi\geq 0$ with $\|\phi\|_{L^2(\mathbb{R}^N)}^2=\tau_1$ and for any number $\theta>0,$ put $u(x)=\theta^{1/2}\phi(\theta^{1/N} x).$
Then, for all $n\in \mathbb{N},$ one obtains that
\begin{equation}\label{SUBAt1}
E_\tau\leq e_2\left(u_2^n \right)+ \frac{1}{2}\theta^{2\alpha/N}\|(-\Delta)^{\alpha/2}\phi\|_{L^2(\mathbb{R}^N)}^2 -\theta^{(p_1-2)/2}\gamma_1\|\phi\|_{L^{p_1}(\mathbb{R}^N)}^{p_1}.
\end{equation}
Now by selecting $\theta$ sufficiently small, one can obtain
\[
\frac{1}{2}\theta^{2\alpha/N}\|(-\Delta)^{\alpha/2}\phi\|_{L^2(\mathbb{R}^N)}^2 -\theta^{(p_1-2)/2}\gamma_1\|\phi\|_{L^{p_1}(\mathbb{R}^N)}^{p_1} <0.
\]
It then follows from \eqref{SUBAt1} that
$
E_\tau< \LimInf{n\to \infty}e_2\left(u_2^n\right),
$
this contradicts the inequality obtained above in \eqref{conlem}. The
proof that $\|u_2^n\|_{L^{p_2}(\mathbb{R}^N)}^{p_2}\geq \delta_2$ follows the same argument.

\smallskip

To prove statement (iii), let $\lambda>1.$ By Lemma~\ref{BounMin}, since the $L^{p_i}$-norms are
bounded away from zero for all sufficiently large $n$, one has that
\begin{equation*}
\begin{aligned}
e_1\left(\lambda f_1^n \right)
& = \lambda^2 e_1\left(f_1^n\right)+\left(\lambda^2-\gamma_1\lambda^{p_1}\right)\|f_1^n\|_{L^{p_1}(\mathbb{R}^N)}^{p_1} \\
& \leq \lambda^2 e_1\left(f_1^n\right)+\left(\lambda^2-\gamma_1\lambda^{p_1}\right) \delta_2
< \lambda^2 e_1\left(f_1^n\right).
\end{aligned}
\end{equation*}
The proof of the scaling property $e_2\left(\lambda f_2^{n}\right)<\lambda^2 e_1\left(f_2^{n}\right)$ is similar. \hfill{$\Box$}

\smallskip

\begin{lem}  \label{Lvanish}
Let $N\geq 2$ and $2^\star_{\alpha}=2N/(N-2\alpha).$ Assume $\{w_n\}_{n\geq 1}$ be a bounded sequence of functions in $H^\alpha(\mathbb{R}^N).$
If there is some $R>0$ such that
\begin{equation}
\lim_{n\to \infty} \left( \sup_{y \in \mathbb{R}^N} \int_{y+B_{R}(0)}|w_n(x)|^{2}\ dx\right) = 0,
\label{vanishhypo}
\end{equation}
then one has
$
\Lim{n\to \infty}\|w_n\|_{L^q(\mathbb{R}^N)} = 0
$
for any $2<q<2^\star_{\alpha}.$
\end{lem}
\noindent{\bf Proof.}\ This is a version of Lemma~I.1 of P. L. Lions \cite{Lio2}.
We provide a proof here for the sake of completeness.
Denote
\[
\epsilon_n=\sup_{y\in \mathbb{R}^N}\int_{y+B_R(0)}|w_n|^2\ dx,
\]
so that $\Lim{n\to \infty}\epsilon_n=0.$ Let $2<q<2^\star_{\alpha}=2N/(N-2\alpha).$ For every point $y\in \mathbb{R}^N$ and any number $R>0,$ using the H\"{o}lder inequality, for every $n$ we obtain that
\[
\|w_n\|_{L^q(y+B_R(0))} \leq \|w_n\|_{L^2(y+B_R(0))}^{\lambda}\|w_n\|_{L^{2^{\star}_{\alpha}}(y+B_R(0)) }^{1+\lambda},
\]
where $\lambda$ satisfies
$
\frac{\lambda q}{2}+\frac{(1+\lambda)q}{2^{\star}_{\alpha}}=1.
$
Taking $\lambda q=s,$
it then follows from the Sobolev inequality that
\begin{equation}\label{sumVan}
\int_{y+B_R(0)}|w_n|^q\ dx \leq C \epsilon_n^{s} \|w_n\|_{L^{2^{\star}_{\alpha}}(y+B_R(0)) }\vertiii{w_n}_{H^\alpha}^s\leq C\epsilon_n^{s} \|w_n\|_{L^{2^{\star}_{\alpha}}(y+B_R(0)) },
\end{equation}
where $s=2/q.$
Now cover the $n$-space $\mathbb{R}^N$ by $n$-balls of radius $R$ in such a way that each $x\in \mathbb{R}^N$ lies at most $N+1$ of these $n$-balls, then
by summing the inequality \eqref{sumVan} over all $n$-balls in the covering and making another use of the Sobolev inequality, we obtain that
\[
\int_{\mathbb{R}^N} |w_n|^q\ dx \leq (N+1) C\epsilon_n^{s} \|w_n\|_{L^{2^{\star}_{\alpha}}(\mathbb{R}^N) }\leq C \epsilon_n^{s},
\]
which gives the desired result. \hfill{$\Box$}

\smallskip

We require the following result of \cite{Guo} concerning the existence of
minimizers of the energy functional $e_i(f)$ associated with the standard
nonlinear fractional Schr\"{o}dinger equations.

\begin{lem}\label{ScalarL}
Suppose $N\geq 2, 0<\alpha<1,$ and $2<p_1, p_2<2+\frac{4\alpha}{N}$. Let $i\in \{1,2\}$ and let the
functional $e_i:H^\alpha(\mathbb{R}^N)\to \mathbb{C}$ be as defined in \eqref{singleEng}.
Then, for any $\tau_i>0,$ if a sequence $\{u_i^n\}_{n\geq 1}$ in $H^\alpha(\mathbb{R}^N)$ be such that $\|u_i^n\|_{L^2(\mathbb{R}^N)}\to \sqrt{\tau_i}$ as $n\to \infty$ and
\[
\lim_{n\to \infty}e_i(u_i^n)=\inf\{e_i(u):u\in H^\alpha(\mathbb{R}^N),\ \int_{\mathbb{R}^N}|u|^2\ dx=\tau_i \},
\]
then the sequence $\{u_i^n\}_{n\geq 1}$ is compact in $H^\alpha(\mathbb{R}^N)$ up to spatial translations and
the extraction of subsequence.
The limit function $\psi_{\tau_i}$ satisfies
\[
(-\Delta)^{\alpha}Q= \omega Q+\gamma |Q|^{p-2}Q,\ \mathrm{for\ some} \ \omega\in \mathbb{R}.
\]
\end{lem}

In the next few lemmas, we closely follow techniques of \cite{ABnls, [AB11], SBdcds} to prove strict inequalities involving
the minimization problem $(E,M_\sigma)$ as function of constraint variables.
These inequalities will play a key role later to excluding the possibility of
dichotomy for an energy-minimizing sequence while applying the concentration compactness principle.
\begin{lem}\label{SUBApos}
For any $\sigma, \tau\in \mathbb{R}^{+}\times \mathbb{R}^{+},$ one has $E_{\sigma+\tau}<E_{\sigma}+E_{\tau}.$
\end{lem}

\noindent{\bf Proof.}\ \ For any $\sigma, \tau\in \mathbb{R}^{+}\times \mathbb{R}^{+},$ let $\{(u_{1,i}^n, u_{2,i}^n)\}_{n\geq 1}$ be any sequence of functions
in $H^\alpha(\mathbb{R}^N)\times H^\alpha(\mathbb{R}^N)$ satisfying the conditions
\[
\begin{aligned}
 & \lim_{n\to \infty}\|u_{j,1}^n\|_{L^2(\mathbb{R}^N)}=\sqrt{\sigma_j},\ \lim_{n\to \infty}\|u_{j,2}^n\|_{L^2(\mathbb{R}^N)}=\sqrt{\tau_j}\ \textrm{for}\ j=1,2,\\
   & \lim_{n\to \infty}E\left(u_{1,1}^n, u_{2,1}^n\right)=E_\sigma, \  \textrm{and}\ \lim_{n\to \infty}E\left(u_{1,2}^n, u_{2,2}^n\right)=E_\tau.
\end{aligned}
\]
By passing to suitable subsequences, one may assume that the following values exists
\[
\begin{aligned}
& A_1=\frac{1}{\sigma_1}\lim_{n\to \infty} \left(e_1\left(u_{1,1}^n\right)-\int_{\mathbb{R}^N} F\left(u_{1,1}^n,u_{2,1}^n \right)\ dx\right),
B_1=\frac{1}{\sigma_2}\lim_{n\to \infty}e_2 \left(u_{2,1}^n\right),\\
& A_2=\frac{1}{\tau_1}\lim_{n\to \infty} \left(e_1\left(u_{1,2}^n\right)- \int_{\mathbb{R}^N}F\left(u_{1,2}^n,u_{2,2}^n \right)\ dx \right),\
B_2=\frac{1}{\tau_2}\lim_{n\to \infty}e_2 \left(u_{2,2}^n\right).
\end{aligned}
\]
We first consider the case that $A_1<A_2.$ Without
loss of generality, we may assume that $u_{1,i}^n$ and $u_{2,i}^n$ are non-negative.
By a density argument, we may also suppose that $u_{1,i}^n$ and $u_{2,i}^n$ have compact support. We denote
$\widetilde{u}_{2,1}^n(\cdot)=u_{2,1}^n(\cdot-b_{2,1}\kappa),$ where $\kappa$ is some unit vector in $\mathbb{R}^N.$
Choose $b_{2,1}$ such that the supports of $\widetilde{u}_{2,1}$ and $u_{2,2}^n$ are disjoint and let
$u_2^{n}=\widetilde{u}_{2,1}^n+u_{2,2}^n$ in $\mathbb{R}^N.$ Then,
we have that $\Lim{n\to \infty}\|u_2^{n}\|_{L^2(\mathbb{R}^N)} = \sqrt{\sigma_2+\tau_2}.$
Let us denote $q_{1,1}=1+\frac{\tau_1}{\sigma_1}$ and $q_{2,2}=1+\frac{\tau_2}{\sigma_2}.$ Then it is obvious that
\begin{equation}\label{subC1}
E_{\sigma+\tau}\leq \lim_{n\to \infty}E\left((q_{1,1})^{1/2}u_{1,1}^n,u_2^{n} \right).
\end{equation}
Introducing the notation
$
J(u,v)=e_1(u)-\int_{\mathbb{R}^N}F(u,v)\ dx
$
for $u,v\in H^\alpha(\mathbb{R}^N).$
We now make use of the fact that $q_{1,1}>1$ to obtain
\begin{equation}\label{subC2}
\begin{aligned}
\lim_{n\to \infty}& J\left((q_{1,1})^{1/2}u_{1,1}^n, u_2^{n} \right) \leq \lim_{n\to \infty}J\left( (q_{1,1})^{1/2}u_{1,1}^n,\widetilde{u}_{2,1}^n \right)\\
& = q_{1,1} \lim_{n\to \infty}\int_{\mathbb{R}^N}\left(\frac{1}{2}|D^\alpha u_{1,1}^n |^2-q_{1,1} \gamma_1|u_{1,1}^n|^{p_1}-F\left( u_{1,1}^n,\widetilde{u}_{2,1}^n \right) \right)\ dx\\
& \leq q_{1,1}\lim_{n\to \infty} J\left(u_{1,1}^n, \widetilde{u}_{2,1}^n\right)=q_{1,1} \sigma_1A_1=\sigma_1A_1+\tau_1A_2-\delta,
\end{aligned}
\end{equation}
where $\delta = \tau_1(A_2-A_1).$ Since $A_1<A_2,$ it is obvious that $\delta>0.$ Applying \eqref{subC2} into \eqref{subC1}, one can
easily deduce that
\begin{equation*}
E_{\sigma+\tau} \leq \lim_{n\to \infty} \left(E\left( u_{1,1}^n, \widetilde{u}_{2,1}^n\right)+E\left( u_{1,2}^n,u_{2,2}^n\right) \right)-\delta
 < E_\sigma+E_\tau.
\end{equation*}
The proof in the case $A_1>A_2$ follows the same argument except that we swap the indices
and so will
not be repeated here.
Next, suppose that $A_{1}=A_{2}$ and $B_{1}\leq B_{2}.$ Invoking Lemma~\ref{BounMin}(iii), there exists $\delta>0$ such that
\begin{equation*}
\begin{aligned}
E_{\sigma+\tau}& \leq E\left((q_{1,1})^{1/2}u_{1,1}^n, (q_{2,2})^{1/2}u_{2,1}^n \right)\\
&\leq q_{2,2} e_2\left(u_{2,1}^n\right)
+J\left((q_{1,1})^{1/2}u_{1,1}^n,(q_{2,2})^{1/2}f_{2,1}^n \right)-\delta \\
& \leq  q_{2,2} e_2\left(u_{2,1}^n\right)+q_{1,1} J\left(u_{1,1}^n,u_{2,1}^n \right)-\delta\\
& = E\left(u_{1,1}^n,u_{2,1}^n \right) + \frac{\tau_2}{\sigma_2} e_2\left(u_{2,1}^n\right)+\frac{\tau_1}{\sigma_1} J\left(u_{1,1}^n,u_{2,1}^n \right)-\delta.
\end{aligned}
\end{equation*}
Passing the limit as $n\to \infty$ on both sides of the preceding inequality and making use of the facts
$A_1=A_2$ and $B_1\leq B_2,$ one obtains that
\begin{equation*}
\begin{aligned}
E_{\sigma+\tau}& \leq E_\sigma+\frac{\tau_2}{\sigma_2}(\sigma_2B_1)+\frac{\tau_1}{\sigma_1}(\sigma_1A_1) -\delta \\
& \leq  E_\sigma +\tau_2B_2+\tau_1A_2-\delta
 < E_\sigma+E_\tau.
\end{aligned}
\end{equation*}
The proof in the case $A_1=A_2$ and $B_1\geq B_2$ follows a similar argument. \hfill{$\Box$}

\begin{lem}\label{SUBApos2}
For any $\tau\in \mathbb{R}^{+}\times \mathbb{R}^{+}$ and $\sigma=(0,\sigma_2)$ with $\sigma_2>0,$ one has
\begin{equation}\label{s1zero}
E_{\sigma+\tau}<E_\sigma+E_\tau.
\end{equation}
\end{lem}
\noindent{\bf Proof.}\ \ Suppose $i\in \{1,2\}$ and let $\{(u_{1,i}^n, u_{2,i}^n)\}_{n\geq 1}$ be sequence of functions
in $H^\alpha(\mathbb{R}^N)\times H^\alpha(\mathbb{R}^N)$ satisfying the conditions
\[
\begin{aligned}
 & \lim_{n\to \infty}\|u_{1,1}^n\|_{L^2(\mathbb{R}^N)}=0,\ \lim_{n\to \infty}\|u_{2,1}^n\|_{L^2(\mathbb{R}^N)}=\sqrt{\sigma_2},\ \lim_{n\to \infty}E\left(u_{1,1}^n, u_{2,1}^n\right)=E_\sigma, \\
 & \lim_{n\to \infty}\|u_{j,2}^n\|_{L^2(\mathbb{R}^N)}=\sqrt{\tau_j}\ \textrm{for}\ j=1,2,
    \  \textrm{and}\ \lim_{n\to \infty}E\left(u_{1,2}^n, u_{2,2}^n\right)=E_\tau.
\end{aligned}
\]
We look for sequence of functions $(u_1^n, u_2^n)$ in $H^\alpha(\mathbb{R}^N)\times H^\alpha(\mathbb{R}^N)$ satisfying
$\|u_1^n\|_{L^2(\mathbb{R}^N)}\to \sqrt{\tau_1},$ $\|u_2^n\|_{L^2(\mathbb{R}^N)}\to \sqrt{\sigma_2+\tau_2},$ $E(u_1^n, u_2^n)\to E_{\sigma+\tau}$ as $n\to \infty,$ and such that \eqref{s1zero} holds.
As before, one can pass to a subsequence
if necessary and consider the values
\[
\begin{aligned}
& D_1=\frac{1}{\sigma_2}\lim_{n\to \infty}  \left( e_2\left(u_{2,1}^n\right)-\int_{\mathbb{R}^N}F\left(u_{1,2}^n, u_{2,1}^n\right)\ dx \right)\ \textrm{and} \\
& D_2=\frac{1}{\tau_2}\lim_{n\to \infty}  \left( e_2\left(u_{2,2}^n\right)-\int_{\mathbb{R}^N} F\left(u_{1,2}^n,u_{2,2}^n\right)\ dx \right).
\end{aligned}
\]
Assume first that $D_1<D_2.$ With $q_{2,2}$ defined as above, it follows that
\begin{equation*}
\begin{aligned}
E_{\sigma+\tau}& \leq E\left(u_{1,2}^n, (q_{2,2})^{1/2}u_{2,1}^n \right)
= e_1\left(u_{1,2}^n\right)\\
& +q_{2,2}\int_{\mathbb{R}^N}\left(\frac{1}{2}|D^\alpha u_{2,1}^n|^2-q_{2,2} \gamma_2|u_{2,1}^n|^{p_2}
-q_{2,2}F\left(u_{1,2}^n, u_{2,1}^n\right) \right)\ dx\\
& \leq e_1\left(u_{1,2}^n\right)+ q_{2,2}e_2\left(u_{2,1}^n \right)-q_{2,2}\int_{\mathbb{R}^N} F\left(u_{1,2}^n,f_{2,1}^n \right)\ dx.
\end{aligned}
\end{equation*}
Put $\delta = \tau_2(D_2-D_1).$ Since $D_1<D_2,$ we have that $\delta>0.$
Passing the limit as $n\to \infty$ in the preceding inequality and using the definition of $q_{2,2},$ we obtain that
\begin{equation*}
\begin{aligned}
E_{\sigma+\tau}& \leq \lim_{n\to \infty} E\left(u_{1,1}^n, u_{2,1}^n\right)+\lim_{n\to \infty} e_1\left(u_{1,2}^n\right)+\frac{\tau_2}{\sigma_2}(\sigma_2D_1)\\
& =E_\sigma+\lim_{n\to \infty} e_1\left(u_{1,2}^n\right) +\tau_2D_2-\delta \\
& =E_\sigma+\lim_{n\to \infty}E\left(u_{1,2}^n, u_{2,2}^n\right)-\delta < E_\sigma+E_\tau.
\end{aligned}
\end{equation*}
Next consider the case that $D_1>D_2.$ With $p_{2,2}=1+\frac{\sigma_2}{\tau_2},$ we have that
\begin{equation*}
\begin{aligned}
E_{\sigma+\tau}& \leq E\left(u_{1,2}^n, (p_{2,2})^{1/2}u_{2,2}^n\right)
=e_1\left(u_{1,2}^n \right)\\
& + p_{2,2}\int_{\mathbb{R}^N}\left(\frac{1}{2}|D^\alpha u_{2,2}^n|^2-p_{2,2} \gamma_2|u_{2,2}^n|^{p_2}
-p_{2,2}F\left(u_{1,2}^n, u_{2,2}^n\right) \right)\ dx\\
& \leq e_1\left(u_{1,2}^n \right)+p_{2,2} e_2\left(u_{2,2}^n \right)-p_{2,2}\int_{\mathbb{R}^N} F\left(u_{1,2}^n,u_{2,2}^n \right)\ dx.
\end{aligned}
\end{equation*}
Put $\delta=\sigma_2(D_1-D_2).$ Then $\delta>0.$ Passing the limit as $n\to \infty$ in the
preceding inequality and using the definition of $p_{2,2},$ it follows that
\begin{equation*}
\begin{aligned}
E_{\sigma+\tau}& \leq \lim_{n\to \infty}E(u_{1,2}^n,u_{2,2}^n )+\frac{\sigma_2}{\tau_2}(\tau_2D_2)= E_{\tau}+\sigma_2D_1-\delta \\
& \leq  E_{\tau} + \lim_{n\to \infty}E\left(u_{1,1}^n, u_{2,1}^n\right)-\delta< E_{\tau}+E_{\sigma}.
\end{aligned}
\end{equation*}
Finally, suppose that $D_1=D_2.$ Let $q_{2,2}$ be as defined in the proof of Lemma~\ref{SUBApos} and $u_2^n=(q_{2,2})^{1/2}u_{2,1}^n.$
Then, using Lemma~\ref{BounMin}(iii), one can find $\delta>0$ such that
\begin{equation*}
E_{\sigma+\tau}\leq E\left(u_{1,2}^n, u_2^n \right)\leq e_1\left(u_{1,2}^n \right)+q_{2,2}e_2\left(u_{2,1}^n \right)-q_{2,2}\int_{\mathbb{R}^N} F\left(u_{1,2}^n,u_{2,1}^n\right)\ dx-\delta.
\end{equation*}
Since $q_{2,2}=1+\frac{\tau_2}{\sigma_2},$ one can pass limit as $n\to \infty$ on both sides of the last inequality and
make use of the fact $D_1=D_2$ to obtain
\begin{equation*}
\begin{aligned}
E_{\sigma+\tau}& \leq E_\sigma+\lim_{n\to \infty}e_1\left(u_{1,2}^n\right)+\frac{\tau_2}{\sigma_2}(\sigma_2D_1)-\delta \\
& = E_\sigma+\lim_{n\to \infty}e_1\left(u_{1,2}^n\right)+\tau_2D_2-\delta \\
& \leq E_\sigma+\lim_{n\to \infty}E\left(u_{1,2}^n,u_{2,2}^n \right)-\delta<E_\sigma+E_\tau.
\end{aligned}
\end{equation*}
This completes the proof of the inequality \eqref{s1zero} in all
three possible cases according to values of $D_1$ and $D_2.$ \hfill{$\Box$}

\smallskip

One can follow the same argument as in the proof of \eqref{s1zero} to prove the following version of subadditivity inequality.
\begin{lem}\label{SUBApos3}
For any $\sigma\in \mathbb{R}^{+}\times \mathbb{R}^{+}$ and $\tau\in \mathbb{R}^{+}\times \{0\},$ one has
$
E_{\sigma+\tau}<E_\sigma+E_\tau.
$
\end{lem}

We are now able to establish the following version of the subadditivity condition.
\begin{lem} \label{subadd}
For all $\sigma, \tau\in (\mathbb{R}^{+}\times \mathbb{R}^{+}) \cup \{\mathbf{0}\}$ with
$\sigma+\tau=\beta \in \mathbb{R}^{+}\times \mathbb{R}^{+}$ and $\sigma, \tau \neq \{\mathbf{0}\},$ one has
\begin{equation}\label{SUBA}
E_\beta<E_\sigma +E_\tau.
\end{equation}
\end{lem}
\noindent{\bf Proof.}\ \ To prove \eqref{SUBA}, we consider
four separate cases: (i) $\sigma, \tau\in \mathbb{R}^{+}\times \mathbb{R}^{+}$;
(ii) $\tau\in \mathbb{R}^{+}\times \mathbb{R}^{+}$ and $\sigma\in \{0\}\times \mathbb{R}^{+}$;
(iii) $\sigma \in \mathbb{R}^{+}\times \mathbb{R}^{+}$ and $\tau\in \mathbb{R}^{+}\times \{0\}$; and (iv) $\sigma\in \{0\}\times \mathbb{R}^{+}$ and
$\tau\in \mathbb{R}^{+}\times \{0\}.$ All other remaining cases can be reduced to one of cases above by swapping the indices. In view of
lemmas \ref{SUBApos}, \ref{SUBApos2}, and \ref{SUBApos3}, it only remains to prove \eqref{SUBA} in the case when $\sigma\in \{0\}\times \mathbb{R}^{+}$ and $\tau\in \mathbb{R}^{+}\times \{0\}.$ Assuming $\tau_2=\sigma_2$ in Lemma~\ref{ScalarL}, let $\psi_{\tau_1}$
and $\psi_{\sigma_2}$ be minimizers of $e_1(u)$ and $e_2(u)$ over $S_{\tau_1}$ and $S_{\sigma_2},$ respectively.
Then it is obvious
that $\int_{\mathbb{R}^N}F(\psi_{\tau_1},\psi_{\sigma_2})\ dx>0$ and \eqref{SUBA} holds. \hfill{$\Box$}

\begin{lem}\label{revLem}
For any $\sigma\in \mathbb{R}^{+}\times \mathbb{R}^{+},$ let $\{(u_1^n, u_2^n)\}_{n\geq 1}$ be a minimizing
sequence for the problem $(E, \mathcal{S}_\sigma).$ Define $Q_n:[0,\infty)\to [0,\sigma_1+\sigma_2]$ by
\begin{equation}\label{Pndef}
Q_n(t)=\sup_{y\in \mathbb{R}}\int_{y+B_t(0)}\left(|u_1^n(x)|^2+|u_2^n(x)|^2\right)\ dx,\ n\geq 1,\ t>0.
\end{equation}
Then the following assertions hold:

\smallskip

(i) Every sequence of non-decreasing
functions $(Q_n)$ has a subsequence converging pointwise
to a non-decreasing function $Q:[0,\infty)\to [0,\sigma_1+\sigma_2].$

\smallskip

(ii) If $\gamma$ is defined by
$\gamma = \Lim{t \to \infty}Q(t),$
then there exists an ordered pair $\tau=(\tau_1, \tau_2) \in [0,\sigma_1]\times [0,\sigma_2]$ such that $\gamma=\tau_1+\tau_2$ and
\begin{equation}\label{RSadd}
E_\sigma\geq E_\tau +E_{\sigma-\tau}.
\end{equation}
\end{lem}
\noindent{\bf Proof.}\ \ Each function $Q_n$ is non-decreasing on $[0,\infty)$ and by Helly's selection
theorem $Q(t)=\Lim{t\to \infty}Q_n(t)$ (perhaps passing to a subsequence) is a
non-decreasing function on $[0,\infty).$ It is easy to check that $0\leq \gamma \leq \sigma_1+\sigma_2.$
To prove statement (ii), let $\epsilon>0$ be an arbitrary.
We continue to assume that the subsequence associated with $\gamma$ is the whole sequence.
It follows from the definition of $\gamma$ that
there exists $t_\epsilon>0$ and $N_\epsilon\in \mathbb{N}$
such that for every $t\geq t_\epsilon$ and $n\geq N_\epsilon,$ one has $\gamma-\epsilon<Q(t)\leq Q(2t)<\gamma$ and
$
\gamma-\epsilon < Q_{n}(t) \leq Q_{n}(2t) \leq \gamma +\epsilon.
$
Thus, by the definition of the concentration functions $Q_{n},$ for every $n\geq N_\epsilon$ there exists a sequence of points $\{y_n\}\subset \mathbb{R}^N$ such that
\begin{equation}\label{sub113}
\int_{y_n+B_{t}(0)}\rho_{n}(x)\ dx> \gamma-\epsilon\ \ \textrm{and}\
\int_{y_n+B_{2t}(0)}\rho_{n}(x) \ dx < \gamma +\epsilon,
\end{equation}
where $\rho_{n}=|u_1^n|^2+|u_2^n|^2.$ Now, for any
$\eta > 0,$ we set
$\rho_\eta(x)=\rho(x/\eta)$ \ \text{and} \ $\sigma_\eta(x)=\sigma(x/\eta),$ where
$\rho \in C_{0}^{\infty }(B_2(0))$ be such that $\rho \equiv 1$ on $B_1(0)$ and
$\sigma \in C^{\infty}(\mathbb{R}^N)$ be such that $\rho^2 + \sigma^2 \equiv 1$ on $\mathbb R^N$.
Next, define the functions
\[
\begin{aligned}
   &\left(u_1^{1,n}(x), u_1^{2,n}(x)\right)=\rho_{t}(x-y_k)\left(u_{1}^n(x),u_2^n(x)\right),\ x\in \mathbb{R}^N,\\
    & \left(u_2^{1,n}(x), u_{2}^{2,n}(x)\right)=\sigma_{t}(x-y_k)\left(u_1^n(x),u_2^n(x)\right), \ x\in \mathbb{R}^N.
\end{aligned}
\]
Then, one can pass to a subsequence to find the numbers
$\tau_1 \in [0,\sigma_1]$ and $\tau_2 \in
[0,\sigma_2]$ such that $\|f_1^{i,n}\|_{L^2(\mathbb{R}^N)}\to \sqrt{\tau_i}$ for $i=1,2,$ whence it also follows 
immediately that $\|f_2^{i,n}\|_{L^2(\mathbb{R}^N)}\to \sqrt{\sigma_i-\tau_i}$ for $i=1,2.$
Now taking into account of these and making use of the inequalities \eqref{sub113}, it is easy to check that
$|(\tau_{1}+\tau_2)-\gamma| <\epsilon .$
Suppose for now that the following holds:
\begin{equation} \label{e12ineq}
E\left(u_1^{1,n},u_1^{2,n}\right)+E\left(u_2^{1,n},u_{2}^{2,n}\right) \le E\left(u_1^n,u_2^n\right) + C\epsilon,\ \ \forall n.
\end{equation}
\noindent To prove the inequality \eqref{RSadd}, since for any given $\epsilon>0,$ each of the
terms in both sides of \eqref{e12ineq}
is bounded independently of $n,$ thus up to a subsequence, one may assume that
$E\left(u_1^{1,n},u_1^{2,n}\right) \to \Lambda_1$ and $E\left(u_2^{1,n},u_{2}^{2,n}\right) \to \Lambda_2.$
In turn, it follows that
\[
\Lambda_1 +\Lambda_2 \leq E_\sigma+C\epsilon.
\]
Since the number $\epsilon$ can be chosen arbitrarily small and $t$ can be taken arbitrarily large, taking account
into the results obtained in the preceding paragraphs, one sees that for every $m\in \mathbb{N},$
one can find sequences of functions $\left(u_{1,m}^{1,n},u_{1,m}^{2,n}\right)$ and $\left(u_{2,m}^{1,n},u_{2,m}^{2,n}\right)$ in
$H^\alpha(\mathbb{R}^N)\times H^\alpha(\mathbb{R}^N)$ such that $\|u_{1,m}^{i,n}\|_{L^2(\mathbb{R}^N)}\to \sqrt{\tau_i(m)}$, $\|u_{2,m}^{i,n}\|_{L^2(\mathbb{R}^N)}\to \sqrt{\sigma_i-\tau_i(m)},$ and $ E\left(u_{i,m}^{1,n},u_{i,m}^{2,n}\right) = \Lambda_i(m)$ for $i=1,2,$
where $\tau_1(m) \in [0,\sigma_1], \ \tau_2(m) \in [0,\sigma_2],$
\begin{equation}\label{Limit1m}
|\tau_1(m)+\tau_2(m)-\gamma| \leq \epsilon,\ \ \textrm{and}\ \ \Lambda_1(m) +\Lambda_2(m) \leq E_\sigma+\frac{1}{m}.
\end{equation}
Passing to a subsequence, we can further suppose
that $\tau_1(m) \to \tau_1\in [0,\sigma_1], \tau_2(m) \to \tau_2 \in [0,\sigma_2], \Lambda_1(m) \to \Lambda_1,$
and $\Lambda_2(m)\to \Lambda_2.$ Moreover, by redefining the
sequences $\left(u_1^{1,n},u_1^{2,n}\right)$ and $\left(u_2^{1,n},u_2^{2,n}\right)$ to be diagonal subsequences
$\left(u_1^{1,n},u_1^{2,n}\right)=\left(u_{1,n}^{1,n},u_{1,n}^{2,n}\right)$ and $\left(u_2^{1,n},u_2^{2,n}\right)=\left(u_{2,n}^{1,n},u_{2,n}^{2,n}\right),$
one can assume that $\|u_{1}^{i,n}\|_{L^2(\mathbb{R}^N)}\to \sqrt{\tau_i}$, $\|u_{2}^{i,n}\|_{L^2(\mathbb{R}^N)}\to \sqrt{\sigma_i-\tau_i},$ and
$E\left(u_i^{1,n},u_i^{2,n} \right)\to \Lambda_i$ for $i=1,2.$
Now, letting the limit as $m \to \infty$ in the first inequality of \eqref{Limit1m}, one obtains $\gamma=\tau_1+\tau_2.$
The condition \eqref{RSadd} follows from the second inequality in \eqref{Limit1m} provided one can show that
\begin{equation}\label{GeqInq}
\Lambda_1 \geq E_\tau\ \ \textrm{and}\ \ \Lambda_2 \geq E_{\sigma-\tau}.
\end{equation}
To see the first inequality in \eqref{GeqInq}, assume first that both $\tau_1$ and $\tau_2$ are positive and define
\[
\beta_1^n=\frac{(\tau_1)^{1/2}}{\|u_1^{1,n}\|_{L^2(\mathbb{R}^N)}}\ \ \textrm{and}\ \ \beta_2^n=\frac{(\tau_2)^{1/2}}{\|u_1^{2,n}\|_{L^2(\mathbb{R}^N)}}.
\]
Then, it is obvious that $E\left(\beta_1^nu_1^{1,n},\beta_2^nu_1^{2,n}\right)\geq E_\tau.$ Since
scaling factors all tend to $1$ as $n\to \infty,$ we have that $E\left(\beta_1^nu_1^{1,n},\beta_2^nu_1^{2,n}\right)\to \Lambda_1$ and hence, the first inequality in \eqref{GeqInq} follows. Next, suppose that one of $\tau_1$ or $\tau_2$ is zero, say $\tau_1=0.$ Then, it follows
from the Sobolev interpolation inequality that
$
\int_{\mathbb{R}^N}|u_1^{1,n}|^{r_1} |u_1^{2,n}|^{r_2}\ dx \to 0\ \textrm{as}\ \ n\to \infty
$
and hence, we obtain that
\begin{equation*}
\Lambda_1  =\lim_{n\to \infty}E\left(u_1^{1,n},u_1^{2,n}\right)
  = \lim_{n\to \infty}\left( e_2\left(u_1^{2,n} \right)+ \int_{\mathbb{R}^N}|D^\alpha u_1^{1,n}|^2\ dx\right) \geq E_{\tau_2}.
\end{equation*}
This concludes the proof of the first inequality in \eqref{GeqInq}.
The proof of second inequality in \eqref{GeqInq} uses 
the same argument with $\sigma_1-\tau_1$ and $\sigma_2-\tau_2$ enjoying the roles of $\tau_1$ and $\tau_2,$ respectively.

\smallskip

To complete the proof of lemma, it only remains to
prove the condition \eqref{e12ineq}. We will make use of the following commutator estimates result.
\begin{lem}\label{CEst}
If $0<\alpha<1$ and $f,g\in \mathcal{S}(\mathbb{R}^N),$ then
\[
\|\left[D^\alpha,f\right]g\|_{L^2(\mathbb{R}^N)}\leq C \left(\|\nabla f\|_{L^{p_1}(\mathbb{R}^N)}\|D^{\alpha-1}g\|_{L^{q_1}(\mathbb{R}^N)}+\|D^\alpha f\|_{L^{p_2}(\mathbb{R}^N)}\|g\|_{L^{q_2}(\mathbb{R}^N)} \right),
\]
where $\left[X, Y \right]=XY-YX$ is the commutator, $q_1, q_2\in [2,\infty),$ and
\[
\frac{1}{p_1}+\frac{1}{q_1}=\frac{1}{p_2}+\frac{1}{q_2}=\frac{1}{2}.
\]
\end{lem}
This lemma is proved in \cite{Guo} by combining the commutator estimates
established by Kato and Ponce in (Lemma~X1 of \cite{Kato}) with a version of Kenig, Ponce, and Vega's result
in (Lemma~2.10 of \cite{Kato2}).

\smallskip

We now show that the condition \eqref{e12ineq} holds.
For ease of notation, let us write the shifted functions $\rho_{t}(x-y_k)$ and $\sigma_{t}(x-y_k)$
simply as $\rho_{t}$ and $\sigma_{t}$ respectively.
By Lemma~\ref{CEst} with $f=\rho_t$ and $g=u_{1}^{1,n}$, we estimate
\[
\begin{aligned}
 \|[D^\alpha,\rho_t]& u_{1}^{1,n}\|_{L^2(\mathbb{R}^N)} =\|D^\alpha(\rho_tu_{1}^{1,n})-\rho_tD^\alpha u_{1}^{1,n}\|_{L^2(\mathbb{R}^N)}\\
& \leq C \left(\|\nabla \rho_t\|_{L^{p_1}(\mathbb{R}^N)}\|D^{\alpha-1}u_{1}^{1,n}\|_{L^{q_1}(\mathbb{R}^N)}
+\|D^\alpha \rho_t\|_{L^{p_2}(\mathbb{R}^N)}\|u_{1}^{1,n}\|_{L^{q_2}(\mathbb{R}^N)} \right).
\end{aligned}
\]
Taking $p_1=\infty, q_1=2; p_2=2+\frac{N}{\alpha},$ and $q_2=2+\frac{4\alpha}{N}$ so that $\frac{1}{p_1}+\frac{1}{q_1}=\frac{1}{p_2}+\frac{1}{q_2}=\frac{1}{2},$ and using the Sobolev inequality, we obtain that
\[
\begin{aligned}
\|[D^\alpha,\rho_t]u_{1}^{1,n}\|_{L^{2}(\mathbb{R}^N)}& \leq C \|\nabla \rho_t\|_{\infty}\|D^{\alpha-1}u_{1}^{1,n}\|_{L^{2}(\mathbb{R}^N)}\\
& +C \frac{1}{t^\mu}\|D^\alpha \rho\|_{L^{p_2}(\mathbb{R}^N)}\vertiii{u_{1}^{1,n}}_{H^\alpha(\mathbb{R}^N )},
\end{aligned}
\]
where $\mu=2\alpha^2/(2\alpha+N).$ By the definition of $\rho_t,$ we
have $\|\nabla \rho_{t}\|_\infty=\|\nabla \rho\|_\infty / t.$ It immediately follows from this
last inequality that
\begin{equation}\label{CEst2}
\int_{\mathbb{R}^N}|D^{\alpha}(\rho_t u_{1}^{1,n})|^2\ dx\leq \int_{\mathbb{R}^N}(\rho_t)^2|D^\alpha u_{1}^{1,n}|^2\ dx +C\varepsilon.
\end{equation}
holds for all sufficiently large $t.$
Using the estimate for $\int_{\mathbb{R}^N}|D^\alpha(\rho_t u_{1}^{1,n})|^2\ dx$ obtained in \eqref{CEst2} and a similar estimate for $\int_{\mathbb{R}^N}|D^\alpha(\rho_t u_{1}^{2,n})|^2\ dx$, we can write
\begin{equation*}
\begin{aligned}
E\left(u_1^{1,n},u_1^{2,n}\right)& \leq \int_{\mathbb{R}^N} \rho _{t}^{2}\left(|D^\alpha u_1^{n}|^2+|D^\alpha u_2^{n}|^2-\gamma_1|u_1^{n}|^{p_1}-\gamma_2|u_2^{n}|^{p_2} - F(u_1^{n}, u_2^{n}) \right)\ dx\\
&\ \ +\int_{\mathbb{R}^N}\left(\gamma_1 U_\rho^{p_1}|u_1^{n}|^{p_1}+\gamma_2 U_\rho^{p_2}|u_2^{n}|^{p_2}+U_\rho^{r_1+r_2}F(u_1^{n}, u_2^{n})\right)\ dx+C\varepsilon,
\end{aligned}
\end{equation*}
where $U_{\rho}^{r}=\rho_t^{2}-|\rho_t|^{r}.$
A similar estimate holds for the quantity $E\left(u_2^{1,n},u_2^{2,n}\right).$
Since $\rho_t$ and $\sigma_t$ satisfy $\rho_{t}^2(x) + \sigma_{t}^2(x) \equiv 1$ for $x\in \mathbb{R}^N,$ we obtain that
\begin{equation}\label{CEst4}
\begin{aligned}
E\left(u_1^{1,n},u_1^{2,n}\right)& +E\left(u_2^{1,n},u_2^{2,n}\right)\leq
E\left(u_1^{n},u_2^{n}\right)+\gamma_1 \int_{\mathbb{R}^N}U^{p_1}|u_1^{n}|^{p_1}\ dx \\
\ &+\gamma_2 \int_{\mathbb{R}^N}U^{p_2}|u_2^{n}|^{p_2}\ dx+\int_{\mathbb{R}^N}U^{r_1+r_2}F(u_1^{n}, u_2^{n})\ dx+C\varepsilon ,
\end{aligned}
\end{equation}
where $U^r$ is given by $U^r=U_{\rho}^{r}+ U_{\sigma}^{r}.$ Now denote $\mathcal{D}=B(y_k,2r)-B(y_k,r).$
Then, using the inequalities we obtained in \eqref{sub113}, it immediately follows that
\[
\begin{aligned}
& \int_{\mathbb{R}^N}U^{p_i}|u_i^{n}|^{p_i}\ dx\leq 4\int_\mathcal{D}|u_i^{n}|^{p_i}\ dx\leq C\epsilon,\ i=1,2,\\
& \int_{\mathbb{R}^N}U^{r_1+r_2}F(u_1^{n}, u_2^{n})\ dx\leq 4\beta \|u_1^{n}\|_{\infty}^{r_1}\int_\mathcal{D} |u_2^{n}|^{r_2}\ dx\leq C\epsilon,
\end{aligned}
\]
where the letter $C$ represents various positive constants independent of $t$ and $n.$ Taking into account all these inequalities,
\eqref{e12ineq} follows from \eqref{CEst4}. \hfill{$\Box$}

\section{Proof of existence result}\label{existence}

Armed with all preliminaries lemmas, we are now able to prove the existence theorem.
To begin, consider any energy-minimizing sequence $\{(u_1^n, u_2^n)\}_{n\geq 1}$ for $E_\sigma.$
Define $Q_n$ and $\gamma$ as in Lemma~\ref{revLem}.
Denoting the subsequence associated to $\gamma$ again by $\{(u_1^n, u_2^n)\}_{n\geq 1},$ suppose for now that
$\gamma=\sigma_1+\sigma_2$ (called the case of compactness).
Then for every $k\in \mathbb{N}$ there exists $t_k\in \mathbb{R}$ and points $y_n\in \mathbb{R}^N$ such that for all sufficiently large $n,$
\begin{equation}\label{ExiEs1}
\|u_1^n\|_{L^2(y_n+B_{t_k}(0))}^2+\|u_2^n\|_{L^2(y_n+B_{t_k}(0))}^2  > (\sigma_1+\sigma_2)-\frac{1}{k}.
\end{equation}
For ease of notation, let us denote the shifted sequences $w_1^n(x)=u_1^n(x+y_n)$ and $w_2^n(x)=u_2^n(x+y_n)$ for $x\in \mathbb{R}^N.$
Then using \eqref{ExiEs1}, for every $k\in \mathbb{N},$ we have that
\begin{equation}\label{ExiEs2}
\|w_1^n\|_{L^2(B_{t_k}(0) )}^2+\|w_2^n\|_{L^2(B_{t_k}(0) )}^2  > (\sigma_1+\sigma_2)-\frac{1}{k}.
\end{equation}
Since the translated sequence $\{(w_1^n, w_2^n)\}$ is bounded uniformly in the space $H^\alpha(\mathbb{R}^N)\times H^\alpha(\mathbb{R}^N),$ so
by taking a subsequence if necessary, one may assume that $(w_1^n, w_2^n)\rightharpoonup (u_1,u_2)$ in $H^\alpha(\mathbb{R}^N)\times H^\alpha(\mathbb{R}^N).$ 
Invoking the Fatou's lemma as $n\to \infty$, we have that $\|u_1\|_{L^2(\mathbb{R}^N)}^2+\|u_2\|_{L^2(\mathbb{R}^N)}^2\leq \sigma_1+\sigma_2.$
For each $k\in \mathbb{N},$ the inclusion of $H^\alpha(B_{t_k}(0))$ into $L^2(B_{t_k}(0))$ is compact, so up to a subsequence,
we may assume that $(w_1^n, w_2^n)\to (u_1,u_2)$ strongly in $L^2(B_{t_k}(0))\times L^2(B_{t_k}(0)).$ Moreover, using a
standard diagonalization technique, one may assume that a single subsequence of 
the translated sequence $\{(w_1^n, w_2^n)\}_{n\geq 1}$ has been
chosen which enjoys this property for every $k.$ Now,
by passing the limit as $n\to \infty$ on both sides of \eqref{ExiEs2}, we obtain that
\[
\begin{aligned}
\|u_1\|_{L^2(\mathbb{R}^N)}^2+\|u_2\|_{L^2(\mathbb{R}^N)}^2  &  \geq \|u_1\|_{L^2( B_{t_k}(0)) }^2+\|u_2\|_{L^2( B_{t_k}(0)) }^2  \\
 & =\lim_{n\to \infty}\left(\|w_1^n\|_{L^2( B_{t_k}(0)) }^2+\|w_2^n\|_{L^2( B_{t_k}(0)) }^2 \right)
  \geq (\sigma_1+\sigma_2)-\frac{1}{k}.
\end{aligned}
\]
Since $\|u_1\|_{L^2(\mathbb{R}^N)}^2+\|u_2\|_{L^2(\mathbb{R}^N)}^2\leq \sigma_1+\sigma_2$ and $k\in \mathbb{N}$ was arbitrary, 
this last inequality in turn implies that
$\|u_1\|_{L^2(\mathbb{R}^N)}^2+\|u_2\|_{L^2(\mathbb{R}^N)}^2= \sigma_1+\sigma_2.$ Thus $(w_1^n, w_2^n)\to (u_1,u_2)$ strongly
in $L^2(\mathbb{R}^N)\times L^2(\mathbb{R}^N).$
Now, using the interpolation inequality for $L^p$-norms and the Sobolev inequality, we have that
\begin{equation}\label{p1con}
\begin{aligned}
\|w_1^n-u_1\|_{L^{p_1}(\mathbb{R}^N)} & \leq \|w_1^n-u_1\|_{L^2(\mathbb{R}^N)}^{\lambda}\|w_1^n-u_1\|_{L^{2^\star_{\alpha}}(\mathbb{R}^N) }^{1-\lambda}\\
& \leq C \|w_1^n-u_1\|_{L^2(\mathbb{R}^N)}^{\lambda}~\vertiii{w_1^n-u_1}_{H^\alpha(\mathbb{R}^N) }^{1-\lambda} \leq C \|w_1^n-u_1\|_{L^2(\mathbb{R}^N)}^{\lambda},
\end{aligned}
\end{equation}
where $\lambda$ satisfies $\frac{1}{p_1}=\frac{\lambda}{2}+\frac{1-\lambda}{2^\star_{\alpha}}.$ Making use of the fact that
$w_1^n\to u_1$ strongly in $L^2(\mathbb{R}^N)$ and the estimate \eqref{p1con}, we obtain
that $\|w_1^n\|_{L^{p_1}(\mathbb{R}^N)}\to \|u_1\|_{L^{p_1}(\mathbb{R}^N)}$ as $n\to \infty.$
Similarly, we have that $\|w_2^n\|_{L^{p_2}(\mathbb{R}^N)}\to \|u_2\|_{L^{p_2}(\mathbb{R}^N)}$ as $n\to \infty.$
We also have
\begin{equation}\label{FLimit}
\int_{\mathbb{R}^N}|w_1^n|^{r_1}|w_2^n|^{r_2}\ dx \to \int_{\mathbb{R}^N}|u_1|^{r_1}|u_2|^{r_2}\ dx
\end{equation}
as $n\to \infty.$ By Lemma~\ref{BounMin}, since the translated sequence $\{(w_1^n, w_2^n)\}_{n\geq 1}$ is bounded
in $H^\alpha(\mathbb{R}^N)\times H^\alpha(\mathbb{R}^N),$
\eqref{FLimit} can be proved by writing
\[
\begin{aligned}
\int_{\mathbb{R}^N}\left(|w_1^n|^{r_1}|w_2^n|^{r_2}-|u_1|^{r_1}|u_2|^{r_2} \right)\ dx & \leq \int_{\mathbb{R}^N}|w_1^n|^{r_1}\left(|w_2^n|^{r_2}-|u_2|^{r_2}\right)\ dx \\
& +\int_{\mathbb{R}^N}|u_2|^{r_2}\left(|w_1^n|^{r_1}-|u_1|^{r_1}\right)\ dx.
\end{aligned}
\]
and estimating separately the limiting behavior of the integrals on the right-hand side as $n\to \infty.$
Now, invoking the Fatou's lemma once again, we obtain that
\begin{equation}\label{Lcon}
E_\sigma = \lim_{n\to \infty}E\left(w_1^n, w_2^n \right)\geq E(u_1,u_2),
\end{equation}
whence $E(u_1,u_2)=E_\sigma.$ Thus $(u_1,u_2)$ is a minimizer for $E_\sigma.$
Finally, since equality in \eqref{Lcon} in fact implies that
\[
\lim_{n\to \infty}\left( \|D^\alpha w_1^n\|_{L^2(\mathbb{R}^N)}^{2}+\|D^\alpha w_2^n\|_{L^2(\mathbb{R}^N)}^{2} \right) = \|D^\alpha u_1\|_{L^2(\mathbb{R}^N)}^{2}+\|D^\alpha u_2\|_{L^2(\mathbb{R}^N)}^{2},
\]
and therefore, $(w_1^n, w_2^n)\to (u_1,u_2)$ strongly in $H^\alpha(\mathbb{R}^N)\times H^\alpha(\mathbb{R}^N).$
Thus, in order to complete the proof of existence theorem, one needs to show that $\gamma=\sigma_1+\sigma_2$ is the only possibility. To see this,
we claim that the following holds for any energy-minimizing sequence:

\smallskip

(i) $\gamma >0$

\smallskip

(ii) $\gamma \not\in (0, \sigma_1+\sigma_2).$

\smallskip

The possibility $\gamma=0$ is called the case of vanishing and
the possibility $\gamma\in (0, \sigma_1+\sigma_2)$ is called the dichotomy. To prove $\gamma>0,$ we argue by contradiction.
If $\gamma=0$ and $\{(u_1^n, u_2^n)\}_{n\geq 1}$ is the subsequence associated with $\gamma,$ then for all $R>0,$ one has
\[
\lim_{n\to \infty}\left( \sup_{y\in \mathbb{R}^N}\int_{y+B_R(0) }\left(|u_1^n(x)|^2+ |u_2^n(x)|^2\right)\ dx\right) =0.
\]
By the part (i) of Lemma~\ref{revLem}, the sequences $\{|u_1^n|\}$ and $\{|u_2^n|\}$ are both bounded in the space $H^\alpha(\mathbb{R}^N).$
Since $2<p_1, p_2<2+\frac{4\alpha}{N}<\frac{2N}{N-2\alpha},$ Lemma~\ref{Lvanish} proves
that $\|u_1^n\|_{L^{p_1}(\mathbb{R}^N )}^{p_1}\to 0$ and $\|u_2^n\|_{L^{p_2}(\mathbb{R}^N)}^{p_2}\to 0$ as $n\to \infty.$
Since $2<2r_1, 2r_2<\frac{2N}{N-2\alpha},$ using the H\"{o}lder inequality and another use
of Lemma~\ref{Lvanish}, it also follows that
\[
\int_{\mathbb{R}^N} |f_1^n|^{r_1}|f_2^n|^{r_2}\ dx  \leq \left( \int_{\mathbb{R}^N} |f_1^n|^{2r_1}\ dx\right)^{1/2} \left( \int_{\mathbb{R}^N} |f_2^n|^{2r_2}\ dx\right)^{1/2}\to 0
\]
as $n\to \infty.$ Taking account into these convergence properties, we obtain that the infimum of the energy satisfies
\[
E_\sigma=\lim_{n\to \infty}E\left(f_1^n,f_2^n \right)\geq \liminf_{n\to \infty}\int_{\mathbb{R}^N}\frac{1}{2}\left(|D^\alpha f_1^n|^2+ |D^\alpha f_2^n|^2 \right)\ dx\geq 0,
\]
which contradicts the Lemma~\ref{NegInfE} and hence, $\gamma>0.$

\smallskip

Next we show that dichotomy is also not an option for an energy-minimizing sequence.
Suppose that $\gamma\in(0,\sigma_1+\sigma_2)$ holds. Let $\tau=(\tau_1,\tau_2)$ be the
same ordered pair that was found in the part (iii) of Lemma~\ref{revLem} and define $\rho$ by $\rho=\sigma-\tau.$
Then, we have that $\tau+\rho=\beta \in \mathbb{R}^{+}\times \mathbb{R}^{+}$ and $\tau, \rho \neq \{\mathbf{0}\}.$
Consequently, Lemma~\ref{subadd} obtains that $E_\beta<E_\tau+E_\rho.$ On the other hand, the part (iii) of
Lemma~\ref{revLem} implies that $E_\sigma\geq E_\tau+E_{\sigma-\tau},$ which is same as
$E_\beta \geq E_\tau + E_\rho,$ a contradiction. Therefore, $\gamma\in(0,\sigma_1+\sigma_2)$ can not occur here.
This completes the proof of Theorem~\ref{existence}.

\bigskip

The proof of part (iv) of the existence theorem follows a similar argument as in Theorem~1.2(v) of \cite{[AB11]} and we do not repeat here.

\section{Proof of stability result}\label{stability}

The stability result can be proved by a
classical argument, which we repeat for completeness. If the claim were not true, then there would exist
a number $\varepsilon>0,$ a sequence $\{(\Psi_{1}^n(0), \Psi_{2}^n(0))\}_{n\geq 1}\subset H^\alpha(\mathbb{R}^N)\times H^\alpha(\mathbb{R}^N),$ and $t_n\geq 0$ such that
for all $n,$
\[
\inf_{(u_1,u_2)\in V(\tau)}\vertiii{(\Psi_{1}^n(0), \Psi_{2}^n(0))-(u_1,u_2)}_{\alpha}<\frac{1}{n},
\]
and
\[
\inf_{(u_1,u_2)\in V(\tau)} \vertiii{(\Psi_{1}^n(t_n), \Psi_{2}^n(t_n))-(u_1,u_2)}_{\alpha}\geq \varepsilon,
\]
where $(\Psi_{1}^n(t_n), \Psi_{2}^n(t_n))$ is the solution of \eqref{Fnls}
emanating from $(\Psi_{1}^n(0), \Psi_{2}^n(0)).$
Since $(u_1,u_2)\in V(\tau),$ one has that
\[
E\left( \Psi_{1}^n(0), \Psi_{2}^n(0)\right) \to E_\tau=E(u_1,u_2),\ \| \Psi_{j}^n(0)\|_{L^2(\mathbb{R}^N)}\to \sqrt{\tau_j},\ j=1,2,
\]
as $n\to \infty.$ Since the energy $E(f_1,f_2)$ and $\int_{\mathbb{R}^N}|f_j|^2\ dx$ both are conserved functionals, therefore
the sequence of functions
\[
\{(\Psi_{1}^n(t_n), \Psi_{2}^n(t_n))\}_{n\geq 1}\subset H^\alpha(\mathbb{R}^N)\times H^\alpha(\mathbb{R}^N)
\]
is a minimizing sequence for the problem $(E,\mathcal{S}_\tau)$. By the existence theorem, this sequence
must be relatively compact in $H^\alpha(\mathbb{R}^N)\times H^\alpha(\mathbb{R}^N)$ up to a translation.
Hence, there exists a subsequence $\{(\Psi_{1}^{n_k}(t_{n_k}), \Psi_{2}^{n_k}(t_{n_k}))\}_{k\geq 1},$
a sequence of points $x(t_{n_k})\in \mathbb{R}^N,$ and an element $(u_1,u_2)\in V(\tau)$ such that
\[
\inf_{(u_1,u_2)\in V(\tau)} \vertiii{(\Psi_{1}^{n_k}(t_{n_k}, \cdot+x(t_{n_k})), \Psi_{2}^{n_k}(t_{n_k}, \cdot+x(t_{n_k}))-(u_1,u_2)}_{\alpha}\to 0
\]
as $k\to \infty,$ which is a contradiction, and hence the set $V(\tau)$ is stable. \hfill{$\Box $}

\end{document}